\documentclass[11pt,a4paper]{amsart}
\usepackage{fullpage}
\usepackage{bbm,color}
\usepackage{amssymb,amscd,hyperref}
\usepackage[all]{xy}
\newtheorem{thm}{Theorem}[section]
\newtheorem{lem}[thm]{Lemma}
\newtheorem{prop}[thm]{Proposition}
\newtheorem{cor}[thm]{Corollary}
\theoremstyle{definition}
\newtheorem{rem}[thm]{Remark}

\newtheorem{exs}[thm]{Examples}

\usepackage{xspace}
\usepackage{color}

\def\algspt{algebra\xspace}
\def\algspta{algebras\xspace}

\def\Z{\mathbb{Z}}
\def\Zp{\mathbb{Z}_{(p)}}

\def\F{\mathbb{F}}
\def\N{\mathbb{N}}
\def\bA{\bar{\mathcal{A}}}
\def\ev0{\mathit{Ev}_0}
\def\tmf{\mathsf{tmf}}

\def\leq{\leqslant}
\def\geq{\geqslant}

\def\ra{\rightarrow}
\def\THH{\mathsf{THH}}
\def\HH{\mathsf{HH}}

\def\ie{\emph{i.e.}}

\def\Tor{\mathsf{Tor}}

\begin{document}
\title{
Towards an understanding of ramified extensions of structured ring spectra}
\author{Bj{\o}rn Ian Dundas}
\address{Department of Mathematics, University of Bergen, Postboks
  7800, 5020 Bergen, Norway}
\email{dundas@math.uib.no}
\urladdr{http://folk.uib.no/nmabd/}
\author{Ayelet Lindenstrauss}
\address{Mathematics Department, Indiana University, 236 Rawles Hall,
  Bloomington, IN 47405, USA}
\email{alindens@indiana.edu}
\urladdr{http://pages.iu.edu/~alindens/}
\author{Birgit Richter}
\address{Fachbereich Mathematik der Universit\"at Hamburg,
Bundesstra{\ss}e 55, 20146 Hamburg, Germany}
\email{richter@math.uni-hamburg.de}
\urladdr{http://www.math.uni-hamburg.de/home/richter/}

\date{\today}

\begin{abstract}
We propose topological Hochschild homology as a tool for measuring
ramification of maps of structured ring spectra. 
We determine second order topological Hochschild homology of the
$p$-local integers.
For the tamely ramified
extension of the map from the connective Adams summand to $p$-local
complex topological K-theory we determine
the relative topological Hochschild homology and show that it detects
the tame ramification of this extension. 
We also determine relative topological Hochschild homology for the
complexification map from connective real to complex topological
K-theory and for some quotient maps with commutative quotients.
\end{abstract}
\maketitle
\section{Introduction}
For a $G$-Galois extension of number fields $K \subset L$ the associated
extension of rings of integers $\mathcal{O}_K \ra \mathcal{O}_L$ will
not be unramified in general. Greither shows in \cite[Chapter 0,
Theorem 4.1]{greither} that the ramification of such an extension can
be detected with the help of the map 
\begin{equation} \label{eq:ramified}
h \colon \mathcal{O}_L \otimes_{\mathcal{O}_K} \mathcal{O}_L \ra 
\prod_G\mathcal{O}_L.
\end{equation}
Here, $h$ is defined as $h(b_1
\otimes b_2) = (b_1g(b_2))_{g \in G}$ for $b_1, b_2 \in
\mathcal{O}_L$. The extension is unramified if $h$ is an
isomorphism. For more general extensions of commutative rings this
still gives an adequate notion of ramification. The 
Hochschild homology of $\mathcal{O}_L$ over $\mathcal{O}_K$ is an
invariant that behaves 
differently depending on whether the extension is tamely
or wildly ramified. For instance 
$\HH^{\Z}(\Z[i])$ is a square-zero extension of $\Z[i]$ with additive
two-torsion in positive odd degrees. 

In the following we consider cohomology theories with a 
multiplicative structure that can be represented by a commutative
monoid object in one of the symmetric monoidal categories of spectra,
for instance the one presented in \cite{EKMM}. The representing
objects are called commutative ring spectra. Examples of such
cohomology theories are singular cohomology with coefficients in a
commutative ring, topological (real or complex) K-theory, and
several cobordism theories. There is an analogue of Hochschild homology
in the context of ring spectra, topological Hochschild homology. It
was defined by B\"okstedt \cite{Boe} and a 
published account can for instance be found in \cite[Chapter IX]{EKMM}. 

Let $A$ be a commutative ring spectrum and let $B$ be a commutative
$A$-\algspt with an action of a finite group $G$ via maps of
commutative
$A$-\algspta. Then the extension $A \ra B$ is called
unramified \cite[(4.1.2)]{R}, if the map
\begin{equation}\label{eq:unramified}
h \colon B \wedge_A B \ra \prod_G B
\end{equation}
is an equivalence. 
Here, $h$ is the analogue of \eqref{eq:ramified} in the context of
spectra.

Rognes shows \cite[9.2.6,
proof of 9.1.2]{R} that the
condition for $B$ to be unramified over $A$ ensures that the map from
$B$ to relative topological Hochschild homology
of $B$ over $A$, $\THH^A(B)$, is a weak equivalence. Thus the failure
of the map
$$ B \ra \THH^A(B)$$
to be a weak equivalence is a measure of the ramification of the
extension $A \ra B$. It also makes sense to study $\THH^A(B)$ in more
general situations, for 
instance in the absence of a group action.

Algebraic K-theory of an ordinary commutative ring $R$, $K(R)$,
contains a lot of  arithmetic information about $R$, such as the Picard group
of $R$, its Brauer group and its units. Trace methods have been useful
for studying $K(R)$: There are trace maps
$$ \xymatrix{
{K(R)} \ar[r]^{trc} \ar[dr]_{tr} & {TC(R)} \ar[d]\\
                      & {\THH(R)}
}$$
that allow us to approximate $K(R)$ by invariants that are easier to
compute, by topological Hochschild homology, $\THH(R)$, and by topological
cyclic homology, $TC(R)$. Trace methods work also well for
connective commutative
ring spectra, \ie, commutative ring spectra whose homotopy groups are
concentrated in non-negative degrees.

Galois extensions of commutative
$S$-algebras in the sense of Rognes \cite[4.1.3]{R} are unramified. A
prominent example is given by the complexification map from real to
complex periodic K-theory, $c \colon KO \ra KU$. Here, complex
conjugation on complex vector bundles induces a $C_2$-action on $KU$ by
maps of commutative $KO$-algebra spectra. But a result of Akhil Mathew
\cite[Theorem 6.17]{mathew} tells us that finite Galois extensions of
a connective spectrum are purely algebraic. So taking the connective
cover of the complexification map

$$\xymatrix{
{ko} \ar[r]^(0.4)c \ar[d]_j & {ku} \ar[d]^j \\
{KO}  \ar[r]^(0.4)c & {KU}
}$$
does \emph{not} yield a $C_2$-Galois extension $ko \ra ku$ because
algebraically
$$ko_* = \Z[\eta, y, w]/2\eta, \eta^3, \eta y,
y^2-4w \ra \Z[u] \cong ku_*$$
is certainly not \'etale. (Here, the
degrees are $|\eta|=1$, $|y|=4$, $|w|=8$ and $|u|=2$.) 

For a commutative $A$-algebra $B$  we denote by
$\THH^{[n],A}(B)$ the higher order topological Hochschild homology of
$B$ as a commutative $A$-algebra, \ie,
$$ \THH^{[n],A}(B) = B \otimes \mathbb{S}^n$$
where $(-) \otimes \mathbb{S}^n$ denotes the tensor with the
$n$-sphere in the category of commutative $A$-algebras.
This can be viewed as the realization of the simplicial commutative
$A$-\algspt whose $q$-simplices are given by
$$ \bigsqcup_{x \in \mathbb{S}^n_q} B,$$
where the coproduct is the smash product over $A$.

Higher $\THH$ of the Eilenberg Mac Lane spectra of local number rings
also detects ramification \cite{dlr}, but after we take 
coefficients in the residue field we cannot see the difference anymore between
tame and wild ramification in higher $\THH$.  We offer some partial
results towards calculations of higher $\THH$ with unreduced
coefficients. We calculate second order $\THH$ of the $p$-local integers:
$$ \THH_*^{[2]}(\Zp) \cong \Zp[x_1, x_2, \ldots ]/p^nx_n, x_n^p = px_{n+1}.$$
See Theorem \ref{thm:thhzp}.

It is possible to determine $\THH_*^{[2]}(\Zp)$ additively using that
$H\Z_{(p)}$ can be constructed as a Thom spectrum of a double-loop map
and applying the methods of \cite{BCS,S}. Blumberg, Cohen and
Schlichtkrull identify $\THH(\Zp)$ with $H\Zp \wedge \Omega 
\mathbb{S}^3\langle 3\rangle_+$ \cite[Theorem 3.8]{BCS}. See
\cite[Corollary 1.1]{klang}  
for a calculation of $\THH_*^{[2]}(\Zp)$. However, this
views $H\Zp$ as an $E_2$-spectrum and not as a commutative
$S$-algebra, so with this method the multiplicative structure of
$\THH_*^{[2]}(\Zp)$ cannot be determined. The multiplicative structure
is  essential  
if one aims at a calculation of $\THH_*^{[n]}(\Zp)$ for larger $n$. 

We study the examples of the connective covers of the
Galois extensions \cite{R} $KO \ra KU$ and $L_p \ra KU_p$. In the
latter case, the connective cover behaves like an extension of the
corresponding rings of integers.  We
test ramification with relative (higher) topological Hochschild
homology and for $\ell \ra ku_{(p)}$ we see that it looks like tame
ramification (see 
Theorem \ref{thm:kulthh}): $\THH^\ell_*(ku_{(p)})$ is a square zero
extension of $\pi_*ku_{(p)}$ of bounded $u$-exponent. We also
determine relative $\THH$ of the 
complexification map $c \colon ko \ra ku$ (see Theorem \ref{thm:kukothh}). 

Working with structured ring spectra means working in a derived
setting, so quotient maps
can be thought of as extensions. We offer some calculations of
relative $\THH$ in situations
where we kill generators of homotopy groups. We consider a version of
$ku/(p, v_1)$ and
quotients of the form $R/x$ where $x$ is a regular element in $\pi_*(R)$ where
$R$ is a commutative ring spectrum such that $R/x$ is still commutative.

\textbf{Acknowledgement} The last named author thanks the Hausdorff
Research Institute for Mathematics in Bonn for its hospitality during
the Trimester Program \emph{Homotopy theory, manifolds, and field
theories}. She also thanks the Department of Mathematics of the Indiana
University Bloomington for an invitation in the spring of 2016.

\section{Second order $\THH$ of the $p$-local integers}
This section consists of a proof of the following somewhat surprising
result.  In the context of the current paper, this calculation is a
starting point for comparing with future calculations for other rings
of integers.  See Remark~\ref{coincidence} for a discussion of the
fact that the answer agrees with topological Hochschild cohomology.
\begin{thm} \label{thm:thhzp}
For all primes $p$:
$$\THH_*^{[2]}(\Z_{(p)}) \cong \Z_{(p)}[x_1, x_2, \ldots]/p^nx_n, x_n^p - px_{n+1}$$
with $|x_n| = 2p^n$.
\end{thm}
The entire section is devoted to proving this result
For all  primes $p$ the exact sequence
\begin{equation} \label{eq:cofib}
\xymatrix@1{
{\THH_*^{[2]}(\Z_{(p)})} \ar[r]^p & {\THH_*^{[2]}(\Z_{(p)})} \ar[r]^-r
& {\THH_*^{[2]}(\Z_{(p)}, \F_p)} \ar[r]^\delta &
{\Sigma\THH_*^{[2]}(\Z_{(p)})}
}
\end{equation}
is a sequence of $\THH_*^{[2]}(\Z_{(p)})$-modules; in particular,
$\delta$ is a module map.  Furthermore, from \cite{dlr} we have that

\begin{equation} \label{eq:thhzpfp} \THH_*^{[2]}(\Z_{(p)}, \F_p) \cong
  \Gamma_{\F_p}(y) \otimes \Lambda_{\F_p}(z),
\end{equation}
where $|y|=2p$ and $|z|=2p+1$.  We denote the generator
$\gamma_{p^i}(y)$ in the divided power algebra
$\Gamma_{\F_p}(y)$ in degree $2p^{i+1}$  by $y_{p^i}$ and if
$t=t_0+t_1p+\dots+t_np^n$ is the $p$-adic expansion of $t$, then we set
$y_t=y_1^{t_0}y_p^{t_1}
\dots y_{p^{n}}^{t_n}$ with $y_t^p=0$.

By the $\Tor$ spectral sequence,
$$ \Tor_{*,*}^{\THH_*(\Zp)}(\Zp, \Zp) \Rightarrow \THH_*^{[2]}(\Z_{(p)})$$
we know that
$\THH_s^{[2]}(\Z_{(p)})$ is finite $p$-torsion for positive $s$ because
$$\THH_*(\Zp)  = \begin{cases}
\Zp, & *=0, \\
\Zp/i, & *=2i-1, \\
0, & \text{ otherwise.}
\end{cases}$$
By \eqref{eq:thhzpfp} and using the notation introduced below it,
this implies that there are integers $a_1,a_2,\dots$ such that
$$ \THH_s^{[2]}(\Z_{(p)})
\begin{cases}
  0, & 2p\not|\,s, \\
\Z/p^{a_t}\{\tilde {y}_t\}, & s=2pt,
\end{cases}
$$
where the $\tilde{y}_t$ are generators of the given cyclic groups
which are sent to the corresponding generators $y_t$
in $\THH_*^{[2]}(\Zp, \F_p)$.
We will show
\begin{lem} \label{lem:bjorn}
The function $a\colon\N \to \N$ factors over the $p$-adic valuation
$v\colon \N \to \N$,
$a_t=b_{v(t)}$, with $b\colon\N \to \N$ a \emph{strictly} increasing
function with positive values and $b_0=1$.
\end{lem}
\begin{proof}
To this end we use induction on the following statement
$\mathbf{P(n)}$ for positive integers $n$.
The generators $\tilde{y}_t$ are chosen inductively.

$\mathbf{P(n)}$: For positive integers $s$, $t$ such that $v(s)$,
$v(t)$ are less than
$n$ the following properties hold:
\begin{enumerate}
\item[(1)] If $v(s)=v(t)$, then $a_s=a_t$,
\item[(2)] If $v(s)>v(t)$, then $a_s>a_t$,
\item[(3)] If $s=s_0+s_1p+\dots+s_{n-1}p^{n-1}$ is the $p$-adic expansion of $s$
(so that $0\leq s_0,\dots,  s_n<p$), then
$\tilde{y}_s=\tilde{y_1}^{s_0}\dots\tilde{y}_{p^{n-1}}^{s_{n-1}}$,
\item[(4)] If $n>v$, then
  $\tilde{y}_{p^v}=p^{a_{p^v}-a_{p^{v-1}}}\tilde{y}_{p^{v-1}}^p$.
\end{enumerate}

We will repeatedly be considering the cofiber sequence
\eqref{eq:cofib}. In homotopy, the maps are  trivial except in degrees
of the form
$2pt$ (for varying $t$) in which case they are
$$
\xymatrix@1{
0 \ar[r] & {\F_p\{zy_{t-1}\}} \ar[r]^-\delta & {\THH_{2pt}^{[2]}(\Zp)}
\ar[r]^p & {\THH_{2pt}^{[2]}(\Zp)} \ar[r]^-r &  {\F_p\{y_{t}\}} \ar[r]
& 0
}
$$
forcing all the $a_t$ to be positive. For any generator $w$,
$\F_p\{w\}$ denotes the graded vector space generated by
$w$.
Here $r$ is multiplicative and $\delta$ is a
$\THH_*^{[2]}(\Zp)$-module map.  By the surjectivity of $r$
we have that the $y_t$'s can be lifted to integral classes.

\subsubsection*{Establishing $\mathbf{P(1)}$}
\label{sec:p1}

Let $t=1$.  The sequence
$$\xymatrix{0 \ar[r] & \F_p\{z\} \ar[r]^-\delta & \THH_{2p}^{[2]}(\Zp) \ar[r]^p &
\THH_{2p}^{[2]}(\Zp) \ar[r]^-r & \F_p\{y_{1}\} \ar[r] & 0. }
$$
shows that $a_1>0$ and by adjusting $z$ up to a unit we may choose
$\tilde{y}_1$ so that
$\delta(z)=p^{a_1-1}\tilde{y}_1$ and $r(\tilde{y}_1)=y_1$.
In the Tor-spectral sequence we only get a $\Z/p\Z$ in bidegree $(1,
2p-1)$ which survives and shows that $a_1=1$, and so
$\delta(z)=\tilde{y}_1$.

If $1<t<p$ the sequence
$$
\xymatrix{
0 \ar[r] & \F_p\{zy_1^{t-1}\} \ar[r]^-\delta & \THH_{2pt}^{[2]}(\Zp)
\ar[r]^p & \THH_{2pt}^{[2]}(\Zp) \ar[r]^-r & \F_p\{y_{1}^t\} \ar[r] &
0}
$$
gives that $\delta(zy_1^{t-1})=\delta(z)\cdot
\tilde{y}_1^{t-1}=\tilde{y}_1^{t}\neq0$, $p\tilde{y}_1^{t}=0$ and
$r(\tilde{y}_1^{t})=y_1^t\neq 0$.  The last point shows that
$\tilde{y}_1^t$ is not divisible by $p$ and
hence we can choose it as our generator: $\tilde{y}_t =
\tilde{y}_1^t$, and furthermore, this generator is killed by $p$, so $a_t=1$.

If $t=t_0+t_1p$ with $0< t_0 <p$, then the sequence
$$\xymatrix{
0 \ar[r] & \F_p\{zy_1^{t_0-1}y_{t_1p}\} \ar[r]^-\delta &
\THH_{2pt}^{[2]}(\Zp) \ar[r]^p & \THH_{2pt}^{[2]}(\Zp) \ar[r]^-r &
\F_p\{y_{1}^{t_0}y_{t_1p}\} \ar[r] & 0}
$$
gives that
$\delta(zy_1^{t_0-1}y_{t_1p})=\delta(z)\cdot\tilde{y}_1^{t_0-1}\tilde{y}_{t_1p}
= \tilde{y}_1^{t_0}\tilde{y}_{t_1p}\neq0$,
$p\tilde{y}_1^{t_0}\tilde{y}_{t_1p} = 0$ and
$r(\tilde{y}_1^{t_0}\tilde{y}_{t_1p}) = {y_1}^{t_0}{y_{t_1p}}\neq 0$ for
any choice of  a lift
$\tilde{y}_{t_1p}$.  The last point shows that
$\tilde{y}_1^{t_0}\tilde{y}_{t_1p}$ is not divisible by $p$ and hence
we can choose it
as our generator: $\tilde{y}_t=\tilde{y}_1^{t_0}\tilde{y}_{t_1p}$, and
furthermore, this generator is killed by $p$, so $a_t=1$.

Note that we may reconsider our choice of $\tilde{y}_{t_1p}$ later,
and so the exact choice of $\tilde{y}_t$
may still change within these bounds, but the choices of $\tilde{y}_1,
\ldots, \tilde{y}_{p-1}$ remain fixed from now on.
Hence $\mathbf{P(1)(1)}-\mathbf{P(1)(3)}$ are established and as
$\mathbf{P(1)(4)}$ is vacuous we have shown $\mathbf{P(1)}$.

\subsubsection*{Establishing $\mathbf{P(n+1)}$}
\label{sec:pnplus1}
Now, assume $\mathbf{P(n)}$.  First, consider the case $t=p^{n}$. For
$\mathbf{P(n+1)(4)}$ we only have to show that
$$\tilde{y}_{p^n}=p^{a_{p^n}-a_{p^{n-1}}}\tilde{y}_{p^{n-1}}^p,$$
and that $a_{p^n}>a_{p^{n-1}}$.  Consider the sequence
$$
\xymatrix{
0 \ar[r] & \F_p\{zy_1^{p-1}\dots y_{p^{n-1}}^{p-1}\} \ar[r]^-\delta &
\THH_{2p^n}^{[2]}(\Zp) \ar[r]^p &
\THH_{2p^n}^{[2]}(\Zp) \ar[r]^-r & \F_p\{y_{p^n}\} \ar[r] & 0.}
$$

Firstly, by induction we have that
\begin{align*}
  \delta(zy_1^{p-1}\dots
  y_{p^{n-1}}^{p-1})&=\tilde{y}_1\tilde{y}_1^{p-1}\dots
  \tilde{y}_{p^{n-1}}^{p-1}\\
  &=p^{a_p-a_1}\tilde{y}_p\tilde{y}_p^{p-1}\dots \tilde{y}_{p^{n-1}}^{p-1}\\
  &=p^{a_p-a_1}p^{a_{p^2}-a_p}\tilde{y}_{p^2}\tilde{y}_{p^2}^{p-1}\dots
  \tilde{y}_{p^{n-1}}^{p-1}=\dots\\
  &=p^{a_{p^{n-1}}-1}\tilde{y}_{p^{n-1}}^p\neq 0.
\end{align*}
Secondly, $$p\delta(zy_1^{p-1}\cdots
y_{p^{n-1}}^{p-1}) = pp^{a_{p^{n-1}}-1}\tilde{y}_{p^{n-1}}^p =
p^{a_{p^{n-1}}}\tilde{y}_{p^{n-1}}^p = 0.$$
Together this shows that (up to a unit) $\delta(zy_1^{p-1}\cdots
y_{p^{n-1}}^{p-1})=p^{a_{p^n}-1}\tilde{y}_{p^n}$, and that
$\tilde{y}_{p^{n-1}}^p=p^{a_{p^n}-a_{p^{n-1}}}\tilde{y}_{p^n}$, and
since $y_{p^{n-1}}^p=0$ that $a_{p^n}>a_{p^{n-1}}$.

Now, for $\mathbf{P(n+1)(1)}$ and $\mathbf{P(n+1)(2)}$, consider a
general $t$ with $v(t)=n$ and write
$t={t_n}p^n+sp^{n+1}$ with $0<{t_n}<p$.  The exact sequence
$$
\xymatrix{
  \F_p\{zy_1^{p-1}\dots
  y_{p^{n-1}}^{p-1}y_{p^n}^{{t_n}-1}y_{sp^{n+1}}\} \ar[r]^(0.65)\delta
  & \THH_{2pt}^{[2]}(\Zp) \ar[r]^p &
  \THH_{2pt}^{[2]}(\Zp) \ar[r]^r & \F_p\{y_{p^n}^{{t_n}}y_{sp^{n+1}}\}
}
$$
gives that
\begin{align*}
  \delta(zy_1^{p-1}\dots y_{p^{n-1}}^{p-1}y_{p^n}^{{t_n}-1}y_{sp^{n+1}})
&=\tilde{y}_1\tilde{y}_1^{p-1}\dots
\tilde{y}_{p^{n-1}}^{p-1}\tilde{y}_{p^n}^{{t_n}-1}\tilde{y}_{sp^{n+1}}\\
&=p^{a_{p^{n}}-1}\tilde{y}_{p^n}^{{t_n}}\tilde{y}_{sp^{n+1}}\neq 0,
\end{align*}
but $p\delta(zy_1^{p-1}\dots
y_{p^{n-1}}^{p-1}y_{p^n}^{{t_n}-1}y_{sp^{n+1}})=p^{a_{p^{n}}}\tilde{y}_{p^n}^{{t_n}}
\tilde{y}_{sp^{n+1}}=0$
and
$r(\tilde{y}_{p^n}^{{t_n}}\tilde{y}_{sp^{n+1}})=y_{p^n}^{{t_n}}y_{sp^{n+1}}\neq
0$.  Again, the last point shows that
$\tilde{y}_{p^n}^{{t_n}}\tilde{y}_{sp^{n+1}}$ is not divisible by $p$,
and so we may choose
$\tilde{y}_t=\tilde{y}_{p^n}^{{t_n}}\tilde{y}_{sp^{n+1}}$, and
furthermore that this generator is annihilated by
$p^{a_{p^{n}}}$, but not by $p^{a_{p^{n}}-1}$, so that $a_t=a_{p^n}$.

Lastly, by $\mathbf{P(n)(3)}$, we have that if
$s=s_0+s_1p+\dots+s_{n-1}p^{n-1}$ is the $p$-adic expansion of $s$,
then
$\tilde{y}_s=\tilde{y}_1^{s_0}\dots\tilde{y}_{p^{n-1}}^{s_{n-1}}$.  If
$t=s+s_np^n$, then
$r(\tilde{y}_s\tilde{y}_{p^n}^{s_n})=y_t$, so we can choose
$\tilde{y}_t=\tilde{y}_s\tilde{y}_{p^n}^{s_n}$
as desired in $\mathbf{P(n+1)(3)}$.
\end{proof}

\subsubsection*{Background on Bocksteins}

Let $(C_*,\partial)$ be a complex of free abelian groups and assume
$\alpha\in C_n$ has the property that $\alpha\otimes 1$ is a cycle in
$C_*\otimes\F_p$.  That the Bockstein
$\beta_{i-1}[\alpha\otimes 1]$ is defined and equal to zero for some
$i\geq 2$, means that there exist $\gamma\in C_n$ and cycle $\delta\in
C_{n-1}$ so that $\partial(\alpha+p\gamma)=p^i \delta$, and in that
case,
$\beta_{i}[\alpha\otimes 1] =[\delta\otimes 1]$.

Assume we have a short exact sequence of complexes of free abelian
groups $0\to B_*\to C_* \to A_*\to 0$.  Choosing a section in each
degree, we may assume $C_n= A_n\oplus B_n$  for all $n$.  Suppose we
have $a\in A_n$ and $b\in B_n$ so that $[a+b]$ represents a cycle in
$C_*\otimes \F_p$  with  $\beta_{i-1}([(a+b)\otimes 1]) =[0]\in
H_{n-1}(C_*\otimes\F_p)$.  As above, there exist $c\in A_n$, $d\in
B_n$, $e\in A_{n-1}$, $f\in B_{n-1}$  with $e+f\in C_{n-1}$ a cycle
with
$\partial(a+b+p(c+d)) = p^i (e+f)$,
and in that case $\beta_{i} ([(a+b)\otimes 1] )=[(e+f)\otimes 1]$.
Then  if $[e\otimes 1]\neq [0]\in H_{n-1}(A_*\otimes\F_p)$, we get
that  $\beta_{i} ([(a+b)\otimes 1] ) \neq [0]$, since $[(e+f)\otimes
1]
\mapsto [e\otimes 1]\neq [0]$ by the homomorphism induced by the
projection $C_*\to A_*$.

\medskip

More generally, consider a filtered complex $C_*$ of free abelian
groups.   Assume
we have a chain
$a\in E(C_*)^0_{s,t}$ in the associated spectral sequence such that
$[a\otimes 1]$ survives to $E(C_*\otimes\F_p)^\infty_{s,t}$ in the mod
$p$ spectral sequence.  If
we know that the class $[(a+b)\otimes 1]\in H_{s+t}(C_*\otimes \F_p)$
with $b\in F_{s-1}(C_*)$ which
$[a\otimes 1]$ represents in $E(C_*\otimes\F_p)^\infty_{s,t}$
satisfies $\beta_{i-1}([(a+b)\otimes 1]) =[0]\in
H_{s+t-1}(C_*\otimes\F_p)$, but that $d^0(a\otimes 1)
=p^i (e\otimes 1)$ and $[e\otimes 1] \neq [0]\in
E(C_*\otimes\F_p)^1_{s, t-1}$, then $\beta_{i} ([(a+b)\otimes 1] )\neq
[0]$.

\subsubsection*{The $p$-order of the multiplicative generators}
We will  calculate $\THH^{[2]}_*(\Zp)$ by studying its Hurewicz image in
$H_*( \THH^{[2]} (\Zp);\F_p)$, using the model
$$ \THH^{[2]} (\Zp)\simeq B(H\Zp, \THH (\Zp), H\Zp).$$

We use the filtration by simplicial skeleta.
We denote  $H_*(H\Zp; \F_p)$ by $\bA$, and by B\"okstedt,
$$H_*( \THH(\Zp); \F_p) \cong \bA \otimes \F_p[x_{2p}] \otimes
\Lambda[x_{2p-1}],$$
where the augmentation $\THH(\Zp)\to H\Zp$ induces the projection $\bA
\otimes \F_p[x_{2p}] \otimes \Lambda[x_{2p-1}]\to\bA$ sending $x_{2p}$
and $x_{2p-1}$ to zero.
We get that
$$ E^1_{*,*} \cong B (\bA, \bA\otimes\F_p[x_{2p}] \otimes \Lambda[x_{2p-1}],\bA)$$
 is isomorphic to
$$B(\bA, \bA, \bA) \otimes B(\F_p, \F_p[x_{2p}], \F_p) \otimes
B(\F_p,  \Lambda[x_{2p-1}], \F_p),$$
and so its homology is
$$E^2_{*,*} \cong \bA \otimes \Lambda[y_{2p+1}] \otimes \Gamma[y_{2p}]$$
with $y_{2p+1}= 1\otimes x_{2p}\otimes 1$ and $y_{2p}= 1\otimes x_{2p-1}$ and
$y_{2p}^{(a)}= 1\otimes x_{2p-1}^{\otimes a}\otimes 1$.

The dimensions in each total degree in the $E^2$-term account for
$p$-torsion of rank $1$ in each positive dimension
divisible by $2p$, and from knowing $\THH_*^{[2]}(\Zp; \F_p)$
\cite[Theorem 3.1]{dlr} we get that this agrees with the abutment of
the spectral sequence, so it has to collapse at $E^2$.

We use this to prove Theorem~\ref{thm:thhzp}.  By
Lemma~\ref{lem:bjorn}, the only remaining problem is to determine the
order of the $p$-torison in each dimension divisible by $2p$.
\begin{lem} \label{lem:ayelet}
The $p$-torsion in $\THH^{[2]}_{2pt}(\Zp)$ is precisely $\Zp/pt \cong
\Zp/p^{v_p(t)+1}$.
\end{lem}
\begin{proof}
We know from Lemma \ref{lem:bjorn}  that for $t=p^a m$, $(p, m)=1$, in
dimension $2pt$ the order of the torsion
is divisible by $p^{a+1}$. We will use the general observation about
Bocksteins above for
$$C_*=C_* (\Omega^\infty ( \THH^{[2]} (\Zp)) ; \Z) =
C_* (\Omega^\infty B(H\Zp, \THH (\Zp), H\Zp); \Z),$$
filtered by simplicial skeleta of the bar construction, to get that
the torsion is exactly $p^{a+1}$.

\medskip
Fixing a $t$, we have two quasi-isomorphisms (letting $s$ vary)
\begin{align*}
C_s  & (\Omega^\infty B_t(H\Zp, \THH (\Zp), H\Zp); \Z) \\
\to
& C_{s+t}  (\Delta^t \times \Omega^\infty B_t(H\Zp, \THH (\Zp), H\Zp),
\partial\Delta^t \times \Omega^\infty B_t(H\Zp, \THH (\Zp), H\Zp); \Z)  \\
\to
&  E^0_{t, s}
\end{align*}
and we call their composition $\varphi$.

We know by B\"okstedt that additively
$\THH (\Zp) \simeq H\Zp \vee \Sigma^{2p-1} H\F_p \vee \cdots,$
so we can map
\begin{align*}
S^0  & \wedge K(\F_p, 2p-1)^{\wedge t} \wedge S^0 =
S^0\wedge (\Omega^\infty (\Sigma^{2p-1} H\F_p ))^{\wedge t} \wedge S^0 \\
\to &
\Omega^\infty H\Zp \wedge (\Omega^\infty (\THH (\Zp) )^{\wedge t}
\wedge \Omega^\infty H\Zp \to
\Omega^\infty(  H\Zp \wedge (\THH (\Zp) ^{\wedge t} \wedge H\Zp).
\end{align*}
We call this composition $\psi$.  It induces
$$\psi_*:\ C_*(K(\F_p, 2p-1)^{\wedge t}; \Z) \to C_* (\Omega^\infty
(H\Zp \wedge \THH (\Zp)^{\wedge t} \wedge H\Zp); \Z),$$
so composing we get a map of complexes
$$\varphi\circ \psi_*:\  C_*( K(\F_p, 2p-1)^{\wedge t}; \Z) \to  E^0_{t, *}.$$
On the Eilenberg Mac Lane space $ K(\F_p, 2p-1)$, we have a $2p$-chain
with integer coefficients $\tilde x_{2p}$ so that $[\tilde x_{2p}]$
(mod $p$) generates $H_{2p}( K(\F_p, 2p-1); \F_p)\cong \F_p$ and
$\partial \tilde x_{2p} = p \tilde x_{2p-1}$ for a chain
$\tilde x_{2p-1}$ so that $[\tilde x_{2p-1}]$ (mod $p$) generates
$H_{2p-1}(K(\F_p, 2p-1); \F_p) \cong \F_p$.  For these elements,
$\beta_1([\tilde x_{2p}]) = [\tilde x_{2p-1}]$.  Note that these
elements map to generators of the stable homology in the correct
dimensions.  Thus, $\varphi\circ \psi_*(\tilde x_{2p})\otimes 1$ can
be taken as a representative of $x_{2p}$, and
$\varphi\circ \psi_*(\tilde x_{2p-1})\otimes 1$ can be taken as a
representative of $x_{2p-1}$, and we still have
$d^0( \varphi\circ \psi_*(\tilde x_{2p}))= p \varphi\circ
\psi_*(\tilde x_{2p-1})$ in $E^0_{1, *}$. And more generally,
for any $a,b\geq 0$, in $E^0_{a+b+1, *}$ we also have
$$d^0( \varphi\circ \psi_*( \tilde x_{2p-1}^{\ \wedge a} \wedge \tilde
x_{2p} \wedge  \tilde x_{2p-1}^{\ \wedge b})) =
p \varphi\circ \psi_*(\tilde x_{2p-1}^{\ \wedge(a+b+1)}).$$

\medskip
We know that the class $(\varphi\circ \psi_*( \tilde x_{2p-1}^{\wedge
  a} \wedge \tilde x_{2p} \wedge  \tilde x_{2p-1}^{\wedge b})) \otimes 1$
represents
 the class $1\otimes x_{2p-1}^{\otimes a} \otimes x_{2p} \otimes
 x_{2p-1}^{\otimes b} \otimes 1$ which survives to $E^2_{a+b+1, *}$
 and therefore to $E^\infty_{a+b+1, *}$, and similarly for
 $(\varphi\circ \psi_*(\tilde x_{2p-1}^{\wedge(a+b+1)}) )\otimes 1$
 and
$1\otimes x_{2p-1}^{\otimes a+b+1} \otimes 1$.

And so, if $t=p^a m$ with $(p, m)=1$,
$$d^0 (\sum_{i=0}^{t-1} (-1)^i (\varphi\circ \psi_*( \tilde
x_{2p-1}^{\wedge i} \wedge \tilde x_{2p} \wedge  \tilde
x_{2p-1}^{\wedge t-1-i}))
\otimes 1 = pt \cdot (\varphi\circ \psi_*(\tilde
x_{2p-1}^{\wedge(t)}))\otimes 1 = p^{a+1} m \cdot
(\varphi\circ \psi_*(\tilde x_{2p-1}^{\wedge(t)}) )\otimes 1. $$
The mod $p$ homology class which is the image under the Hurewicz map
of $z\gamma_{t-1}(y)$ can be expressed as
$$(1\otimes x_{2p} \otimes 1) (1\otimes x_{2p-1}^{\otimes n-1}\otimes 1)$$
via the bar construction and it is represented by
$(\sum_{i=0}^{t-1} (-1)^i ( \varphi\circ \psi_*(\tilde
x_{2p-1}^{\wedge i} \wedge \tilde x_{2p} \wedge  \tilde
x_{2p-1}^{\wedge
  t-1-i})
\otimes 1$.
From Lemma \ref{lem:bjorn} we have a lower bound on the order of the
torsion and hence
$\beta_{a}(z\gamma_{t-1}(y)) = [0]$  and by the $d^0$ calculation
above $\beta_{a+1}(z\gamma_{t-1}(y)) = \gamma_t(y)$ up to a unit.

This result is a result on stable mod $p$ homology rather than on
stable mod $p$ homotopy,
but since we are applying it to the images under the Hurewicz map of
the two stable mod
$p$ homotopy classes of an Eilenberg Mac Lane space of rank $1$ $p$-torsion, the
Bockstein operators have to do the same on the mod $p$ homotopy.
\end{proof}
\begin{proof}[Proof of Theorem \ref{thm:thhzp}]
Set $x_n = \tilde{y}_{p^n}$. Then we get the $p$-order of these
elements from Lemma
\ref{lem:ayelet} and we worked out the multiplicative relations in
Lemma~\ref{lem:bjorn}.
\end{proof}
\begin{rem}\label{coincidence}
Mike Hill noticed that $\THH_*^{[2]}(\Z_{(p)})$ is abstractly
isomorphic to $\THH^*(\Z_{(p)})$: the calculation of
$\THH^*(\Z_{(p)})$ is due to Franjou and Pirashvili \cite{FP}. We are
not sure whether
this is a coincidence or whether (for some commutative $S$-algebras)
there is a duality between $\THH_*^{[2]}$ and topological Hochschild
cohomology. Note, however, that $\THH_*^{[2]}(\F_p)$ is an exterior
algebra over $\F_p$ on a class in degree three whereas $\THH^*(\F_p)$
is much larger:
$$ \THH^*(\F_p) \cong \F_p[e_0, e_1, \ldots]/(e_0^p, e_1^p, \ldots),
\quad |e_i| = 2p^i $$
\cite[7.3]{fls}, \cite{boe}, so there is no isomorphism of these
groups in general.
\end{rem}

\section{Greenlees' approach to $\THH$}
There is a relative version of the cofiber sequence from \cite[Lemma
7.1]{greenlees} already mentioned in \cite{dlr}. We make it explicit
for later use.  Here and elsewhere $S$ denotes the sphere spectrum.
\begin{lem} \label{lem:cofibers}
Let $R$ be a commutative $S$-algebra and let $C \ra B \ra k$ be
a sequence of maps of commutative
$R$-algebras. Then there is a cofiber sequence of commutative
$k$-\algspta
$$ B \wedge^L_C k \ra \THH^R(C,k) \ra \THH^R(B,k).$$
\end{lem}
The proof is obtained from the one of \cite[Lemma 7.1]{greenlees} by
replacing the sphere spectrum by $R$.

\begin{rem}
Note that there are two cofiber sequences for any such sequence $C \ra
B \ra k$, because we can forget the
commutative $R$-algebra structures on $C$ and $B$ and consider them as
commutative $S$-algebras. This
gives a commutative diagram of cofiber sequences
$$ \xymatrix{
{B \wedge_C k} \ar[r] \ar@{=}[d]& {\THH(C,k)} \ar[d] \ar[r] &
{\THH(B,k)} \ar[d]\\
{B \wedge_C k} \ar[r] & {\THH^R(C,k)} \ar[r] & {\,\THH^R(B,k),}
}$$
so $B \wedge_C k$ measures the difference of the absolute and also of the
relative $\THH$-terms of $C$ and $B$.
\end{rem}

Let us abbreviate $B \wedge^L_C k$ by $A$.
Lemma \ref{lem:cofibers} provides an equivalence
$$ \THH^R(B,k) \simeq \THH^R(C,k) \wedge^L_A k$$
and thus we get a spectral sequence whose $E^2$-term is
$$ \Tor_{*,*}^{A_*}(\THH^R_*(C,k), k_*)$$
 which converges to $\THH^R_*(B,k)$.

We will consider the following examples.
\begin{itemize}
\item
Let $\ell$ denote the Adams summand of $p$-local connective topological
complex K-theory,
$ku_{(p)}$, for some odd prime $p$. For
$$ R=\ell \ra C=\ell \ra B=ku_{(p)} \ra k$$
with $k=H\Z_{(p)}$ or $k=H\F_p$ we obtain calculations for
$\THH^\ell_*(ku_{(p)},k)$. We determine $\THH^\ell_*(ku_{(p)})$ by
different means. 

\item
The complexification map from real to complex topological K-theory
$c\colon ko \ra ku$ is a map of commutative
$S$-algebras.
Wood's theorem displays the $ko$-module $ku$ as the cofiber of the
Hopf map $\eta\colon\Sigma ko\to ko$.  Consequently, the $ku$-module
$ku\wedge_{ko}ku$ is the cofiber of $\eta\colon\Sigma ku\to ku$, and
the resulting short exact sequences $$ 0 \ra \pi_{2m}ku \ra
\pi_{2m}(ku \wedge_{ko} ku) \ra
\pi_{2m-1}(\Sigma ku) \ra 0$$
are split via the multiplication map on $ku$, because the map $ku \ra
ku \wedge_{ko} ku$ above is induced by the unit map
 of $ku$ as a commutative $ko$-algebra
 so we get
$$ \pi_{2m}(ku \wedge_{ko} ku) \cong \pi_{2m}ku \oplus \pi_{2m-2}(ku).$$
We will determine the $ku_*$-algebra structure of $\pi_*(ku
\wedge_{ko} ku)$ in Lemma \ref{lem:kukoalg}. This is the input for the
Tor-spectral sequence computing $\THH_*^{ko}(ku)$ and we will identify
$\THH_*^{ko}(ku)$ in Theorem \ref{thm:kukothh}.

We will also use the cofiber sequences of commutative $k$-algebras
$$ ku \wedge_{ko} k \ra ku \ra \THH^{ko}(ku,k)$$
for $k= H\Z_{(2)}$ and $k=H\F_2$ and we will calculate $\THH$
of $ku$ over $ko$ with coefficients in $H\Z_{(2)}$ and $H\F_2$ (see Proposition
\ref{prop:kukohigherthh}).

\item
We propose $ku_{(p)} \wedge_\ell H\F_p$ as a model for $ku/(p, v_1)$
and use the sequence
$$ S \ra H\F_p \ra ku_{(p)} \wedge_\ell H\F_p \ra H\F_p$$
for calculating its $\THH$ with coefficients in $H\F_p$ (Proposition
\ref{prop:thhkumodpv1}).

\item
In Section \ref{sec:quotients} we determine relative topological
Hochschild homology of quotient maps $R \ra R/x$.
\end{itemize}

\section{Relative $\THH$ of $ku_{(p)}$ as a commutative $\ell$-algebra}
Let $p$ be an odd prime.
On the level of coefficients, the map from the connective Adams
summand to $p$-local connective topological complex K-theory
is
$\ell_* = \Z_{(p)}[v_1] \ra \Z_{(p)}[u] = (ku_{(p)})_*$, $v_1
\mapsto u^{p-1}$. The corresponding $p$-complete
periodic extension is a $C_{p-1}$-Galois extension \cite{R}.  However, the
connective extension is not unramified, but it is
a topological analogue of a tamely ramified extension. Rognes defined
a notion of $\THH$-\'etale extensions in \cite[9.2.1]{R}:
A map of commutative $S$-algebras $A  \ra B$ is \emph{formally $\THH$-\'etale},
if the canonical map from $B$ to $\THH^A(B)$ is an
equivalence. For instance, Galois extensions are formally
$\THH$-\'etale \cite[9.2.6]{R}.
We will show that the map $\ell \ra ku_{(p)}$ is not formally $\THH$-\'etale
by determining $\THH^\ell(ku_{(p)})$. Rognes mentions in \cite[p.~59]{R} that
$ku_{(p)} \ra \THH^\ell(ku_{(p)})$ is a $K(1)$-local equivalence and
Sagave showed in \cite{Sa} that the map $\ell \ra ku_{(p)}$ is
log-\'etale. Ausoni
proved that the $p$-completed extension even satisfies Galois descent
for $\THH$ and algebraic
K-theory \cite[Theorem 1.5]{ausoni}:
$$ \THH(ku_p)^{hC_{p-1}} \simeq \THH(\ell_p), \quad K(ku_p)^{hC_{p-1}}
\simeq K(\ell_p). $$

The tame ramification is visible in $\THH$:
\begin{thm} \label{thm:kulthh}
$$\THH_*^\ell(ku_{(p)}) \cong (ku_{(p)})_* \rtimes
((ku_{(p)})_*/u^{p-2}) \langle
y_0,y_1,\ldots \rangle,$$
where $(ku_{(p)})_* \rtimes M$ denotes a square-zero extension of
$(ku_{(p)})_*$ by a $(ku_{(p)})_*$-module $M$.
The degree of $y_i$ is $2pi+3$.
\end{thm}
\begin{proof}
We can apply the B\"okstedt spectral sequence with $\pi_*$ as the
homology theory because $(ku_{(p)})_*$ is projective over $\ell_*$.
The $E^2$-page consists of
$$ E^2_{s,t} = \HH_{s,t}^{\ell_*}((ku_{(p)})_*, (ku_{(p)})_*).$$
As an $\ell_*$-algebra $(ku_{(p)})_*$ is isomorphic to
$\ell_*[u]/(u^{p-1}-v_1)$.  From \cite{LL} we know that we can use
the following complex in order to calculate Hochschild homology:
$$ \xymatrix{ \ldots \ar[r]^(0.4){\Delta(u)}&
  {\Sigma^{2p}(ku_{(p)})_*} \ar[r]^0  &  {\Sigma^{2p-2}(ku_{(p)})_*}
  \ar[r]^{\Delta(u)} &  {\Sigma^2(ku_{(p)})_*} \ar[r]^0
  &{(ku_{(p)})_*,}} $$
where $\Delta(u) = (p-1)u^{p-2}$. As $(p-1)$ and $v_1$ are units in
$\ell_*$, this yields:
$$ \HH_i^{\ell_*}((ku_{(p)})_*, (ku_{(p)})_*) = \begin{cases}
(ku_{(p)})_*, & \text{ if } i=0, \\
\Sigma^{2mp-2m+2}(ku_{(p)})_*/u^{p-2}, & \text{ if } i=2m+1, m \geq 0, \\
0, & \text{ otherwise.}
\end{cases} $$
As $\THH^\ell(ku_{(p)})$ is an augmented commutative
$ku_{(p)}$-algebra, we know that $ku_{(p)}$ splits off
$\THH^\ell(ku_{(p)})$. Therefore the copy of the homotopy groups of
$ku_{(p)}$ in the zero column of the spectral
sequence has to survive and cannot be hit by any differentials. For
degree reasons, there are no other possible non-trivial
differentials and the spectral sequence collapses at the
$E^2$-page.

In every fixed total degree there is only one term in the
$E^2$-page contributing to this degree:
If we consider an element $u^{k_1}$ in homological degree $2m_1+1$ and
another element $u^{k_2}$ in homological degree $2m_2 +1$ for $m_1
\neq m_2$, then their
total degrees are $2m_1p+2k_1+3$ and $2m_2p+2k_2+3$. These degrees can
only be equal if $2p(m_1-m_2) = 2(k_2-k_1)$. Thus $p$ has to divide
$k_2-k_1 \neq 0$. But $0 \leq k_1,k_2  \leq p-3$, so this cannot happen.

Thus there are no additive
extensions and therefore additively we get the desired result.

As $\THH^\ell_*(ku_{(p)})$ is an augmented graded commutative
$(ku_{(p)})_*$-algebra and as everything in the augmentation ideal is
concentrated in odd degrees there cannot be any non-trivial
multiplication of any two elements in the augmentation ideal.

The spectral sequence is a spectral sequence of $(ku_{(p)})_*$-modules
and elements of the form $u^k \cdot \Sigma^{2mp-2m+2}u^m$ are cycles,
thus the copy of $(ku_{(p)})_*$ in homological degree zero acts on
$ku_{(p)})_*/u^{p-2}y_m$ in the standard way.
\end{proof}
\begin{rem}
For Galois extensions of non-connective commutative ring spectra we
would like to have a good notion of rings of integers. In the above
case $ku_{(p)}$ behaves like the ring of integers of $KU_{(p)}$, and
similarly for the connective Adams summand. The result for relative
$\THH$ corresponds to the one of ordinary rings of integers \cite{LM}. In
other cases, taking the connective cover does not seem to give good
results.
\end{rem}
For coefficients in $H\Z_{(p)}$ and $H\F_p$ we obtain a rather different result.
\begin{prop}
$$ \THH_*^\ell(ku_{(p)}, H\Z_{(p)}) \cong
\Lambda_{\Z_{(p)}}(\varepsilon u) \otimes \Gamma_{\Z_{(p)}}(\varphi^0
u)$$
and also
$$ \THH_*^\ell(ku_{(p)}, H\F_p) \cong \Lambda_{\F_p}(\varepsilon u)
\otimes \Gamma_{\F_p}(\varphi^0 u).$$
\end{prop}
\begin{proof}
We consider the sequence of comutative $\ell$-algebras
$$ R=\ell \ra C=\ell \ra B=ku_{(p)} \ra k$$
with $k=H\Z_{(p)}$ and $k=H\F_p$. In both cases we can identify
$\THH^\ell(ku_{(p)}, k)$ with $ku_{(p)} \wedge_\ell k$ 
and get a Tor-spectral sequence
$$ \Tor^{\pi_*(ku_{(p)} \wedge_\ell k)}_{*,*}(\pi_*k,\pi_*k)
\Rightarrow \THH_*^\ell(ku_{(p)},k).$$
For $k = H\Z_{(p)}$ homological algebra tells us that
$$ \Tor^{\Z_{(p)}[u]/u^{p-1}}_{*,*}(\Z_{(p)}, \Z_{(p)}) \cong
\Lambda_{\Z_{(p)}}(\varepsilon u) \otimes \Gamma_{\Z_{(p)}}(\varphi^0
u).$$
Here, $|\varepsilon u| = 3$ and $|\varphi^0 u| = 2p$. There are no
differentials in this spectral sequence for degree reasons and there
are no
multiplicative extensions, hence we get the claim.

For $k= H\F_p$ the same method gives
$$ \THH_*^\ell(ku_{(p)}, H\F_p) \cong \Lambda_{\F_p}(\varepsilon u)
\otimes \Gamma_{\F_p}(\varphi^0 u).$$
\end{proof}
\begin{rem}
At the prime $3$ we get
$$\THH_*^{[0],\ell}(ku_{(3)}, H\F_3)  \cong \F_3[u]/u^2$$
hence with the methods of \cite{dlr} we can deduce that
$$ \THH^{[0],\ell}(ku_{(3)}, H\F_3) \simeq H\F_3 \vee \Sigma^2 H\F_3$$
as an augmented commutative $H\F_3$-algebra and that we can calculate
higher $\THH$
as iterated Tor-algebras. Hence we get
$$ \THH_*^{[n+1],\ell}(ku_{(3)}, H\F_3) \cong
\Tor_{*,*}^{\THH_*^{[n],\ell}(ku_{(3)}, H\F_3)}(\F_3, \F_3)$$
for all $n \geq 0$.

Using Greenlees' spectral sequence \cite[Lemma 3.1]{greenlees} one can
actually deduce
that this is true at all odd primes. See \cite{torus} for related arguments.
\end{rem}

\section{Relative $\THH$ of the complexification
map}
The graded commutative ring $ko_*$ is $\Z[\eta,y,w]/\langle 2\eta,
\eta y, \eta^3, y^2-4w \rangle$ with $|\eta|=1$, $|y|=4$ and $w$ is
the Bott class in degree $8$.
The complexification map $c\colon ko \ra ku$ induces a map
$c_*\colon ko_*  \ra ku_*= \Z[u]$ and it sends $\eta$ to zero, $y$ to
$2u^2$ and the Bott class $w$ to $u^4$.

Note that the homotopy fixed points of $ku$ with respect to complex
conjugation are not equivalent to $ko$. The homotopy fixed points
spectral sequence yields generators in negative degrees in the
homotopy groups of $ku^{hC_2}$ \cite[5.3]{R}.

The published version of this paper unfortunately contained a mistake in a
calculation that resulted in erroneous statements in Lemma 5.1 and
Theorem 5.2. We are grateful to Eva H\"oning who discovered the mistake. 
\begin{lem} \label{lem:kukoalg}
There is an isomorphism of graded commutative augmented $ku_*$-algebras
\[ (ku \wedge_{ko} ku)_* \cong ku_*[s]/(s^2-su)\]
with $|s| = 2$, where the $ku_*$-algebra structure on $(ku \wedge_{ko}
ku)_*$ is from the left and the augmentation is given by the
multiplication $m\colon ku\wedge_{ko}ku\to ku$ and by $s\mapsto 0$. 
\end{lem}
\begin{proof}
  Smashing Wood's cofiber sequence
  $\xymatrix{ko\ar[r]^\iota&ku\ar[r]^-{\j}&\Sigma^2ko}
  $
  with $ku$ (from the left) over $ko$
  gives a split exact sequence (with unit isomorphism $ku\cong
  ku\wedge_{ko}ko$ suppressed) 
  \[ \xymatrix@1{
      0 \ar[r] & \pi_*ku \ar[rr]^-{(1\wedge\iota)_*} & & \ar@/_4ex/[ll]_{m_*}  
\pi_*(ku \wedge_{ko} ku) 
\ar[rr]^-{(1\wedge \j)_*} & & \pi_*\Sigma^2 ku \ar[r] & 0. } \]

Let $u\in \pi_2ku$ be the generator with $\j_*(u)=\Sigma^22\in
\pi_2\Sigma^2ko\cong\Z$. 
Let $u_l$ and $u_r$ be the images of $u$ in $\pi_2(ku\wedge_{ko}ku)$
induced by the left and right inclusion of $ku$ in $ku\wedge_{ko}ku$. 
If $s$ is the unique element in $\pi_2(ku\wedge_{ko}ku)$ with
$(1\wedge\j)_*s=-\Sigma^21$ and $m_*s=0$, then
$(1\wedge\j)_*(u_r+2s)=1\cdot\j_*u-2=0$ and $m_*(u_r+2s)=u$. Since
also $(1\wedge\j)_*u_l=0$ and $m_*u_l=u$ we must have $u_r+2s=u_l$.

In $ku_* \otimes_{ko_*} ku_*$, and hence also in $\pi_4(ku \wedge_{ko} ku)$,  
we have that $2u_r^2 -2u_l^2 = 0$. As $\pi_*(ku \wedge_{ko} ku)$ is 
torsion free, we get $u_r^2-u_l^2 = 0$ and therefore 
\[ u_r^2 -u_l^2 = (u_l-2s)^2-u_l^2 = 4s^2-4su_l = 0. \] 
Again, since there is no torsion, this yields $s^2 - su_l = 0$.   
\end{proof}
\begin{thm} \label{thm:kukothh}
The Tor spectral sequence \[ E^2_{*,*} = \Tor_{*,*}^{(ku \wedge_{ko}
  ku)_*}(ku_*,ku_*) \Rightarrow \THH^{ko}_*(ku)\]
collapses at the $E^2$-page and  $\THH^{ko}_*(ku)$ is a 
square zero extension of $ku_*$:
\[ \THH^{ko}_*(ku) \cong ku_* \rtimes (ku_*/u)\{ y_0, y_1, \ldots
\}\] 
with $|y_j| = (1 + |u|)(2j+1) = 3(2j+1)$.

\end{thm}
\begin{proof}
Lemma~\ref{lem:kukoalg} implies that the $E^2$-term of  the Tor
spectral sequence is
\[ E^2_{*,*} = \Tor_{*,*}^{(ku \wedge_{ko} ku)_*}(ku_*,ku_*) =
\Tor_{*,*}^{ku_*[s]/(s^2-su)}(ku_*,ku_*). \]
We have a periodic free resolution of $ku_*$ as a module over
$ku_*[s]/(s^2-su)$
\[ \xymatrix@1{ {\ldots}\, \ar[r]^(0.25){s} &
    {\Sigma^4ku_*[s]/(s^2-su)} \ar[r]^-{s-u}
    &
  {\Sigma^2ku_*[s]/(s^2-su)} \ar[r]^-{s} &
  {ku_*[s]/(s^2-su).}} \]
Tensoring this down to $ku_*$ yields
\[ \xymatrix@1{{\ldots} \ar[r]^(0.4){0} & {\Sigma^4 ku_*} \ar[r]^{-u}
  & {\Sigma^2 ku_*} \ar[r]^0& {ku_*}.
}\]
As $ku_*$ splits off $\THH^{ko}_*(ku)$, the zero column has to survive
and cannot be hit by differentials and hence all differentials are
trivial.

For the $E^\infty$-term we therefore get $E^\infty_{0,*} \cong ku_*$,
$E^\infty_ {2j,*} =0$ for $j>0$, and $E^\infty_{2j+1,*} \cong (ku_*/u
)\{y_j\}$ for $y_j$ in bidegree $(2j+1, 4j+2)$ if $j>0$.
So even total degrees occur only in $E^\infty_{0,*}$ and odd total
degrees occur only in 
at most one bidegree and we do not need to worry about additive extensions.
As $y_j$ corresponds to $\Sigma^{4j+2} 1 \in \Sigma^{4j+2}ku_*$, the action
of $u^i \in ku_*$ on $y_j$ is $u^iy_j$ and this is trivial for $i \geq 1$ in
$(ku_*/u)\{y_j\}$ so $u^iy_j$ is zero in $E^\infty_{2j+1, 4j+2}$.  But
$E^\infty_{2j+1, 4j+2}$ has all the elements  
of total degree $6j+3$ in the entire $E^\infty $-term, so in fact the
element in  $\THH^{ko}_*(ku)$ that $y_j$ represents is killed by
multiplication by $u^i$ for any $i\geq 1$.  Thus we have no nontrivial
products of the $u^i$, $i\geq 1$,  and the odd dimensional elements. 

Since the elements  of  $\THH^{ko}_*(ku)$ represented by the  $y_i$
are all in odd degrees, if there were nonzero products among them they
would have to be  elements in $E^\infty_{0,*} \cong ku_*$.  But
elements in $ku_*$ are not killed by multiplying by $u$, whereas the
elements represented by the $y_j$ are.  So there can be no such
nontrivial products. 
\end{proof}

We consider the sequence of commutative $ko$-algebras
$ R=ko \ra C=ko \ra B = ku $ with $k = H\F_2$ or $k=H\Z_{(2)}$ and,
(since $\THH^{ko}(ko, k) \simeq k$), we get
cofiber sequences of commutative $k$-algebras
$$ ku \wedge_{ko} k \ra k \ra \THH^{ko}(ku, k).$$
This yields a Tor-spectral sequence
\begin{equation} \label{eq:spsqkokuk}
E^2_{s,t} = \Tor^{\pi_*(ku \wedge_{ko} k)}_{s,t}(k_*, k_*) \Rightarrow
\THH_{s+t}^{ko}(ku, k).
\end{equation}
Wood's cofiber sequence identifies $ku$ as the cone on $\eta$: $\Sigma
ko \ra ko$. Thus we
get a cofiber sequence
$$ \Sigma k \ra k \ra ku \wedge_{ko} k $$
and $\pi_*(ku \wedge_{ko} k) \cong \pi_*(k \vee \Sigma^2 k) \cong
\Lambda_{\pi_*k}(x_2)$ where $x_2$ is a generator of degree two.

For $k = H\F_2$ and $H\Z_{(2)}$ we can deduce with \cite[2.1]{dlr}
that as a commutative augmented
$k$-algebra $ku \wedge_{ko} k$ is weakly equivalent to the square-zero extension
$k \vee \Sigma^2 k$. Thus
$$ \THH^{ko}(ku, k) \simeq k \wedge_{k \vee \Sigma^2 k} k$$
and the spectral sequence \eqref{eq:spsqkokuk} reduces to
$$ E^2_{s,t} = \Tor^{\pi_*k[x_2]/x_2^2}_{s,t}(\pi_*k, \pi_*k) \Rightarrow
\THH_{s+t}^{ko}(ku, k).$$
But $\Tor^{\pi_*k[x_2]/x_2^2}_{s,t}(\pi_*k, \pi_*k) \cong
\Lambda_{\pi_*k}(\varepsilon x_2) \otimes \Gamma_{\pi_*k}(\varphi^0
x_2)$
with $|\varepsilon x_2| = 3$, $|\varphi^0x_2| = 6$, and we know from
\cite{BLPRZ} combined with the methods from \cite[Section 3]{dlr} that
there cannot be any differentials in this spectral sequence. Hence we obtain
\begin{prop} \label{prop:kukohigherthh}
$$ \THH^{ko}_*(ku, H\Z_{(2)}) \cong \Lambda_{\Z_{(2)}}(\varepsilon x_2) \otimes
\Gamma_{\Z_{(2)}}(\varphi^0 x_2)$$
and
$$ \THH^{ko}_*(ku, H\F_2) \cong \Lambda_{\F_2}(\varepsilon x_2) \otimes
\Gamma_{\F_2}(\varphi^0 x_2).$$
Over $\F_2$ we can also iterate the calculation and obtain
 $$ \THH^{[n+1], ko}_*(ku, H\F_2) \cong
 \Tor_{*,*}^{\THH^{[n], ko}_*(ku, H\F_2)}(\F_2, \F_2).$$
\end{prop}

\begin{rem}
To the eyes of $\THH$ with coefficients in $H\F_p$ coefficients
(for $p=2$ resp.~$p=3$) the extensions $ko \ra ku$ and $\ell \ra ku_{(3)}$
show a similar behaviour. This is analogous to the algebraic case:
Hochschild homology homology of the $2$-local Gaussian integers with
coefficients in $\F_2$ is isomorphic to $\Lambda_{\F_2}(x_1) \otimes
\Gamma_{\F_2}(x_2)$ and $\HH^{\Z_{(3)}}_*(\Z_{(3)}[\sqrt{3}], \F_3) \cong
\Lambda_{\F_3}(x_1) \otimes \Gamma_{\F_3}(x_2)$. Thus Hochschild homology
(and also higher Hochschild homology)
with reduced coefficients doesn't distinguish tame from wild
ramification either.
\end{rem}

\section{$ku_{(p)} \wedge_\ell H\F_p$ as a model for $ku/(p,v_1) $}
John Greenlees asks in \cite[Example 8.4]{greenlees} for a commutative
$S$-algebra model of $ku/(p,v_1)$.
We suggest $ku/(p,v_1) = ku_{(p)} \wedge_\ell H\F_p$ which is a
commutative $S$-algebra (even an augmented commutative
$H\F_p$-\algspt, which might not be what Greenlees had in mind) and
satisfies $\pi_*(ku_{(p)} \wedge_\ell H\F_p) \cong \F_p[u]/u^{p-1}$.

\begin{rem}
Alternatively one could  consider $ku/(p,v_1)$ defined by an
iterated cofiber sequence. This is an $A_\infty$-ring spectrum
\cite[3.7]{angeltveit}, hence
an associative $S$-algebra, but we cannot expect any decent level of
commutativity without the price of getting something of the homotopy
type of a generalized Eilenberg-Mac\,Lane spectrum: if $ku/(p,v_1)$
were a pseudo-$H_2$ spectrum, then it
automatically splits as a wedge of suspensions of $H\F_p$'s
\cite[III.4.1]{bmms}. In particular, an $E_\infty$-structure (\ie, a
commutative $S$-algebra structure) would lead to such a splitting.
\end{rem}

We determine $\THH(ku_{(p)} \wedge_\ell H\F_p, H\F_p)$.
\begin{prop} \label{prop:thhkumodpv1}
Topological Hochschild homology of $ku_{(p)} \wedge_\ell H\F_p$ with
coefficients in $H\F_p$ is
$$ \THH_*(ku_{(p)} \wedge_\ell H\F_p, H\F_p) \cong \F_p[\mu] \otimes
\Lambda_{\F_p}(\varepsilon u) \otimes \Gamma_{\F_p}(\varphi^0 u)$$
where $\F_p[\mu] = \THH_*(H\F_p)$.
\end{prop}
\begin{proof}
Greenlees' cofiber sequence \cite[7.1]{greenlees} yields an equivalence
$$\THH(ku_{(p)} \wedge_\ell H\F_p, H\F_p) \simeq H\F_p
\wedge^L_{ku_{(p)} \wedge_\ell H\F_p} \THH(H\F_p).$$
Therefore, the $\Tor$-spectral sequence has $E^2$-term
$$ \Tor^{\F_p[u]/u^{p-1}}_{*,*}(\F_p, \THH_*(H\F_p)).$$
We use the standard periodic resolution of $\F_p$ over
$\F_p[u]/u^{p-1}$. As $\THH(H\F_p)$ has
the same chromatic type as $H\F_p$, $u$ acts by zero on $\THH_*(H\F_p)
= \F_p[\mu]$ and hence the $E^2$-term is isomorphic to
$$ \F_p[\mu] \otimes \Lambda_{\F_p}(\varepsilon u) \otimes
\Gamma_{\F_p}(\varphi^0u).$$
As $\THH(ku_{(p)} \wedge_\ell H\F_p)$ is an augmented commutative
$\THH(H\F_p)$-algebra, the $\F_p[\mu]$-factor splits off and hence
there cannot be any differentials and multiplicative extensions.
\end{proof}

\section{Killing regular generators in $\pi_*R$} \label{sec:quotients}

Killing regular elements in the homotopy groups of a commutative
$S$-algebra rarely gives rise to commutative quotients. However, there
are some important examples for which  we \emph{do} get commutative
quotients whose relative $\THH$ can be calculated.

\begin{prop} \label{prop:thh-regularquotients}
Let $R$ be a connective commutative $S$-algebra whose coefficients
$\pi_*R$ are concentrated in even degrees, with a nonzero divisor $x$
of positive degree. If the canonical map $R \ra R/x$ is a morphism of
commutative $S$-algebras, then the Tor spectral sequence
$$\Tor_{*,*}^{\pi_*(R/x \wedge_R R/x)}(R_*, R_*) \Rightarrow \THH_*^R(R/x)$$
collapses at the $E^2$-term. Its $E^\infty$-term is isomorphic to
$\Gamma_{\pi_*R/x}(\rho^0  \varepsilon x)$ with $|\rho^0  \varepsilon
x| = |x| + 2$ and there are no additive extensions.
\end{prop}
\begin{proof}
The defining cofiber sequence
$$ \xymatrix@1{{\Sigma^{|x|}R} \ar[r]^-x & {R} \ar[r] & {R/x} }$$
gives, via a  $\Tor$-spectral sequence, that
$$\pi_*(R/x \wedge_R
R/x) \cong \Lambda_{\pi_*(R)/x}(\varepsilon x)$$
with $|\varepsilon x| = |x| + 1$. In the spectral sequence for $\THH$
we have as an $E^2$-term
$$ \Tor^{ \Lambda_{\pi_*(R)/x}(\varepsilon x)}_{*,*}(\pi_*R/x,
\pi_*R/x).$$
We consider the periodic resolution of $\pi_*R/x$
$$ \xymatrix@1{{\ldots} \ar[r]^(0.3){\varepsilon x} &
  {\Sigma^{2|x|+2}\Lambda_{\pi_*R/x}(\varepsilon x)} \ar[r]^-{\varepsilon x}  &
  {\Sigma^{|x|+1}\Lambda_{\pi_*R/x}(\varepsilon x)}
  \ar[r]^-{\varepsilon x}  & {\Lambda_{\pi_*R/x}(\varepsilon x)} }$$
and tensor it down to $\pi_*R/x$. As $\pi_*R/x$ is concentrated in
even degrees, the multiplication by $\varepsilon x$ induces the
trivial map and hence our Tor-terms are the homology of the complex
$$ \xymatrix@1{{\ldots} \ar[r]^(0.3){\varepsilon x=0} &
  {\Sigma^{2|x|+2}\pi_*R/x} \ar[r]^-{\varepsilon x=0}  &
  {\Sigma^{|x|+1}\pi_*R/x}  \ar[r]^-{\varepsilon x=0}  & {\pi_*R/x}}$$
and this gives a divided power algebra $\Gamma_{\pi_*R/x}(\rho^0
\varepsilon x)$
with a generator $\rho^0  \varepsilon x$ in degree $|x| + 2$. We have
to show that there are no non-trivial differentials and no extension
problems. The spectral sequence is a spectral sequence of
$\pi_*R/x$-algebras because $R/x$ is assumed to be a commutative
$R$-algebra, hence $\THH^R(R/x)$ is a commutative $R/x$-algebra.

As we assumed that $x$ has positive degree, we can split $\Gamma_{\pi_*R/x}(\rho^0
\varepsilon x)$  as $\pi_*R/x \otimes_{\pi_0R}  \Gamma_{\pi_0R}(\rho^0
\varepsilon x)$. The $\pi_*R/x$-module generators are the
$\pi_0R$-module generators in
$\Gamma_{\pi_0R}(\rho^0 \varepsilon x)$. These generators sit in
bidegrees of the form $(n, n(|x|+1))$.
A differential $d^r$ on a generator in
bidegree $(n, n(|x|+1))$ is in bidegree $(n-r, n(|x|+1) +r -1)$. A
general element in the spectral sequence come from  a product of powers of
generators times an element from $R_*/x$, hence we get that a potential target
has a bidegree of the form
$$ (\sum_i u_i n_i, (\sum_i u_i n_i)(|x|+1) + 2m).$$
Comparing components of the bidegree gives the two equations
$$ n-r = \sum_i u_i n_i \text{ and } n(|x|+1) +r -1 = (|x|+1)(\sum_i
u_i n_i) + 2m.$$
We rewrite the second equation as
$$ 2m+1 = (n - \sum_i u_i n_i)(|x|+1) +r.$$
Using that $n - \sum_i u_i n_i$ is $r$ yields
$2m+1 = r(|x|+2)$, but the degree of $x$ is even, so there can be no
non-trivial differentials in this spectral sequence.

We do not have additive extensions because the $E^\infty$-term is free over
$\pi_*R/x$. Thus as an $\pi_*R/x$-module we get that $\THH^R_*(R/x)$
is isomorphic to $\pi_*R/x \otimes_{\pi_0R}  \Gamma_{\pi_0R}(\rho^0
\varepsilon x)$.
\end{proof}

\begin{cor} \label{cor:rmodxhk}
If in addition to the assumptions in Proposition
\ref{prop:thh-regularquotients} we have
that $R/x$ is an Eilenberg-MacLane spectrum of a commutative ring $k$, then
$$\THH^R(Hk, Hk) \simeq Hk \wedge_{Hk \vee \Sigma^{|x|+1} Hk} Hk$$
as augmented commutative $Hk$-algebras.
In particular,
$$ \THH^R_*(Hk) \cong \Gamma_k(\rho^0\varepsilon x)$$
with $|\rho^0\varepsilon x| = |x|+2$
\end{cor}
\begin{proof}
Greenlees' cofiber sequence identifies $\THH^R(Hk)$ as
$$ Hk \wedge^L_{Hk \wedge_R Hk} Hk$$
using the sequence of commutative ring spectra $R = R \ra Hk = Hk$.
The homotopy groups of  $Hk \wedge_R Hk$ are isomorphic to
$\Lambda_k(\varepsilon x)$
with $|\varepsilon x| = |x|+1$. Hence we know from \cite[Proposition
2.1]{dlr} that
$$ Hk \wedge_R Hk \sim Hk \vee \Sigma^{|x|+1} Hk$$
with the square zero multiplication as augmented commutative
$Hk$-\algspta. Therefore we get the
first claim. This also shows that $\THH^R(Hk)$ can be modeled as the
two-sided bar construction
$$ B^{Hk}(Hk, Hk \vee \Sigma^{|x|+1} Hk, Hk)$$
and by \cite{dlr} we know that its homotopy groups are the homology groups
of the algebraic bar construction $B^k(k, \Lambda(\varepsilon x), k)$.
We know from \cite[Proposition 2.3]{BLPRZ}
that there is a quasiisomorphism between $\Gamma_k(\rho^0\varepsilon
x)$ (with zero differential)
and the differential graded commutative algebra associated to $B^k(k,
\Lambda(\varepsilon x), k)$.
\end{proof}
\begin{cor} If in addition to the assumptions of Corollary \ref{cor:rmodxhk}
the ring $k$ is the field $\F_p$ we get
$$ \THH^{[n+1],R}_*(H\F_p, H\F_p)  \cong \Tor_{*,*}^{\THH^{[n],
    R}_*(H\F_p, H\F_p)}(\F_p, \F_p)$$
for all $n \geq 0$.
\end{cor}
\begin{rem}
In the above statement one can consider a slightly more general case
of any field of characteristic $p$.

\end{rem}

\begin{prop} \label{prop:thh-fieldcoeffs}
Assume in addition to the requirements of Proposition
\ref{prop:thh-regularquotients} that there is a regular sequence $(x,
y_1, \ldots, y_n)$ in $\pi_*R$ such that $R/(x,
y_1, \ldots, y_n)$ is $Hk$ for some field $k$. Then
$$ \THH^{R}_*(R/x, Hk) \cong \Gamma_k(\rho^0 \varepsilon x)$$
with $|\rho^0 \varepsilon x| = |x|+2$.
If $k=\F_p$, then
$$ \THH^{[n+1], R}_*(R/x, H\F_p) \cong \Tor_{*,*}^{\THH^{[n],
    R}_*(R/x, H\F_p)}(\F_p, \F_p)$$
for all $n \geq 0$.
\end{prop}
\begin{proof}
We consider the sequence of commutative $S$-algebras
$$ R \ra R \ra R/x \ra Hk.$$
Then $\pi_*(Hk \wedge_R R/x) \cong \Lambda_k(\varepsilon x)$
and as before we can conclude with \cite[2.1]{dlr} that
$Hk \wedge_R R/x$ is equivalent to the square zero extension
$Hk \vee \Sigma^{|\varepsilon x|} Hk$ in the homotopy category of
commutative augmented $Hk$-\algspta.

Greenlees' cofiber sequence identifies $\THH^R(R/x, Hk)$ as
$$ Hk \wedge^L_{Hk \vee \Sigma^{|\varepsilon x|} Hk} Hk$$
and we know from \cite{dlr,BLPRZ} that this gives $\THH^R_*(R/x, Hk)
\cong \Gamma_k(\rho^0 \varepsilon x)$.

Higher $\THH$ can be calculated using the Tor spectral sequence
associated to the  $2$-sided bar construction: A simplicial model for
$\THH^{[n+1], R}(R/x, H\F_p)$ is
$$ B(H\F_p, \THH^{[n], R}(R/x, H\F_p), H\F_p)$$
and we know from the methods established in \cite{dlr} and
\cite{BLPRZ} that these Tor-spectral sequences
all collapse at the $E^2$-term with no non-trivial extensions.
\end{proof}

\begin{exs}
We end the section with some examples.
\begin{enumerate}
\item
Let $R$ be an Eilenberg-MacLane spectrum $HA$ with $A$ a commutative
ring and let
$x$ be regular in $A$. Then $\THH_*^{HA}(HA/x)$ is isomorphic to
Shukla-homology of $A/x$ over $A$,
$SH_*^A(A/x)$. In this case we obtain
$$ \THH_*^{HA}(HA/x) \cong SH_*^{A}(A/x) \cong \Gamma_{A/x}(\rho^0 \varepsilon x)$$
with $|\rho^0 \varepsilon x| = 2$. An explicit generator
of $SH_{2m}^A(A/x)$ is given by
$$ \sum_{i=0}^m (-1)^i \tau^{\otimes i} \otimes 1 \otimes \tau^{\otimes m-i}.$$
Here, we consider the resolution of $A/x$ that is of the form
$(A[\tau]/\tau^2, d(\tau) = x)$.

\noindent
Higher order Shukla homology is crucial for determining higher order
$\THH$ of $\Z/p^m\Z$ with coefficients in $\Z/p\Z$, see \cite{aim}.
\item
Recall that the connective covers of  the Morava $E$-theories, $e_n$,
have coefficients
$$\pi_*(e_n) \cong W\F_{p^n}[[u_1,\ldots, u_{n-1}]][u]$$ with $|u|=2$,
where $W\F_{p^n}$ denotes the Witt vectors
over $\F_{p^n}$ and where the $u_i$ are generators in degree
zero. Hence $\pi_0(e_n) = W\F_{p^n}[[u_1,\ldots, u_{n-1}]]$.
The
quotient $e_n/u = HW\F_{p^n}[[u_1,\ldots, u_{n-1}]]$ is a commutative
$S$-algebra and the map $e_n \ra e_n/u$ can be realized
as a map of commutative $S$-algebras.

The residue field $H\F_{p^n}$ is the quotient
$e_n/(u, u_1,\ldots, u_{n-1}, p)$ and thus the results of Section
\ref{sec:quotients} allow us to calculate
$\THH^{e_n}_*(e_n/u, e_n/u)$ and  $\THH^{e_n, [m]}_*(e_n/u,
H\F_{p^n})$ for all $m \geq 1$.

\item
Lawson and Naumann show in \cite{LN12} that $BP\langle 2\rangle$ at
the prime two is a commutative $S$-algebra by identifying it with the
$2$-localized connective spectrum of topological modular forms
together with a level three structure, $\tmf_1(3)_{(2)}$. They prove
in \cite[section 5]{LN} that there is a map of commutative
$S$-algebras $\varrho\colon \tmf_1(3)_{(2)} \ra
ku_{(2)}$ and there is a complex orientation of $\tmf_1(3)_{(2)}$ such
that the effect of $\varrho$ on homotopy groups is as follows
\cite[section 5]{LN}:
$$ \pi_*(\tmf_1(3)_{(2)}) = \Z_{(2)}[a_1, a_3] \ra \Z_{(2)}[u], \quad
a_1 \mapsto u, \quad a_3 \mapsto 0.$$
Here the degree of $a_i$ is $2i$.

With  Propositions
\ref{prop:thh-regularquotients} and \ref{prop:thh-fieldcoeffs} we can
determine $\THH^{\tmf_1(3)_{(2)}}_*(ku_{(2)})$ additively and
we get explicit formulae for higher relative $\THH$ of
$ku_{(2)} \cong \tmf_1(3)_{(2)}/a_3$ with respect to $\tmf_1(3)_{(2)}$
and with coefficients in $H\F_2 = \tmf_1(3)_{(2)}/(a_3, a_1, 2)$.

\item
The discretization map from $ku$ to $H\Z = ku/u$ gives rise to another
example of a regular quotient with a commutative $S$-algebra
structure with residue field $H\F_p = ku/(u,p)$ for any prime $p$, and
so does the map from the connective Adams summand $\ell$ to $H\Z_{(p)}
= \ell/v_1$ with residue field $H\F_p = \ell/(v_1, p)$.
Thus in these cases we can determine $\THH^{[n], ku}(H\Z, H\Z/p\Z)$,
$\THH^{[n], ku}(H\Z/p\Z, H\Z/p\Z)$ for all primes and
$\THH^{[n], \ell}(H\Z_{(p)}, H\Z/p\Z)$, $\THH^{[n], \ell}(H\Z/p\Z,
H\Z/p\Z)$ for all odd primes
and all $n$.

\end{enumerate}
\end{exs}

\end{document}